\documentclass[12pt,epsfig,amsfonts]{amsart} 
\usepackage{amsmath,amsthm,amssymb,amscd,epsfig,color}
\usepackage{graphicx}
\usepackage{mathrsfs}

\newtheorem*{theoremA4.1}{Theorem A3.1}

\newtheorem{prop}{Proposition}[section]
\newtheorem{lemma}[prop]{Lemma}

\newtheorem{theorem}[prop]{Theorem}
\newtheorem{cor}[prop]{Corollary}

\theoremstyle{remark}
\newtheorem{remark}[prop]{Remark}

\numberwithin{equation}{section}

\begin{document}

\author{Hiroki Takahasi}

\address{Keio Institute of Pure and Applied Sciences (KiPAS), Department of Mathematics,
Keio University, Yokohama,
223-8522, JAPAN} 
\email{hiroki@math.keio.ac.jp}
\urladdr{\texttt{http://www.math.keio.ac.jp/~hiroki/}}


\title[Uniqueness of minimizer and equidistribution]{
Uniqueness of minimizer for\\
 countable Markov shifts and\\  equidistribution of periodic points} 
 
 \date{\today}
 
 \maketitle
   \begin{abstract}
For a finitely irreducible countable Markov shift and a potential with summable variations,
 we provide a condition on the associated pressure function 
which ensures that
Bowen's Gibbs state, the equilibrium state, and the minimizer of the level-2 large deviations rate function 
are all unique and they coincide.
 From this, we deduce that all periodic points weighted with the potential equidistribute with respect to the Gibbs-equilibrium state
   as the periods tend to infinity.
Applications are given to the Gauss map, and the Bowen-Series map associated with a finitely generated 
free Fuchsian group with parabolic elements.
   \end{abstract}

  \section{Introduction}
The theory of large deviations aims to characterize limit behaviors of measures in terms of rate functions.
 A sequence $\{\mu_n\}_{n=1}^\infty$ of Borel probability measures on a topological space 
 $\mathcal X$ satisfies {\it the Large Deviation Principle} (LDP) if
there exists a lower semi-continuous function 
 $I\colon\mathcal X\to[0,\infty]$ such that
  for every Borel subset $\mathcal B$ of $\mathcal X$
   the following holds:
 \begin{equation*}
-\inf_{{\mathcal B}^o}I\leq\varliminf_{n\to\infty}\frac{1}{n} \log \mu_n(\mathcal B^o)\leq\varlimsup_{n\to\infty}\frac{1}{n} \log \mu_n(\overline{\mathcal B})\leq
-\inf_{\overline{\mathcal B}}I,\end{equation*}
where $ \log 0=-\infty$, $\inf\emptyset=\infty$,
${\mathcal B}^o$ and $\overline{\mathcal B}$ denote the interior and the closure of $\mathcal B$ respectively.
The function $I$ is called {\it a rate function}, and is called {\it a good rate function}
if the level set $\{x\in\mathcal X\colon I(x)\leq\alpha\}$ is compact for every $\alpha\in(0,\infty)$.

If $\mathcal X$ is a metric space, then the rate function is unique.
We call $x\in\mathcal X$ {\it a minimizer} if $I(x)=0$. 
For a closed set $\mathcal B$ of $\mathcal X$ which is disjoint from the set of minimizers,
the LDP ensures that $\mu_n(\mathcal B)$ decays exponentially as $n\to\infty$.
If moreover $I$ is a good rate function,  
the support of any accumulation point
of $\{\mu_n\}_{n=1}^\infty$ is contained in the set of minimizers.
Hence, it is important to determine the set of minimizers.
The non-uniqueness of minimizer is referred to as a phase transition.
The uniqueness of minimizer implies several strong conclusions.

Donsker and Varadhan have identified three levels of the LDP, see \cite[Chapter 1]{Ell85} for details.
This paper is concerned with the level-2 LDP in the thermodynamic formalism.
 For finite topological Markov shifts and H\"older continuous potentials,
unique Gibbs-equilibrium states have been constructed by Bowen \cite{Bow75} and Ruelle \cite{Rue78}.
Extensions of this construction to countable Markov shifts have been done by Mauldin and Urba\'nski \cite{MauUrb03}, Sarig \cite{Sar99,Sar03}
after the works of Walters \cite{Wal78}, Gurevic and Savchenko \cite{GS98}.
The level-2 LDP for finite Markov shifts has been established by Takahashi \cite{Tak84}, Kifer \cite{Kif90,Kif94},
and in some particular cases by
Orey and Pelikan \cite{OrePel89}.

The level-2 LDP for countable Markov shifts has been established in \cite{Tak18}
under the assumption of the existence of a Gibbs state and a strong connectivity of transition matrices.
The assumption on potentials in \cite{Tak18} is rather general, and not much
 is known about minimizers in such a generality.
As the first result of this paper,  in Theorem \ref{minimizer-unique} we provide conditions 
 which ensure the uniqueness of minimizer.

  The uniqueness of minimizer implies several strong conclusions. e.g.,
an exponential decay,  Erd\"os-R\'enyi's law \cite{DenKab07} for the level-1 LDP, the differentiability of pressure
\cite{Ell85}.
The second result of this paper is Theorem \ref{distribution} which asserts that 
the uniqueness of minimizer in Theorem \ref{minimizer-unique} implies a weighted equidistribution 
of periodic points.
For dynamical systems of compact spaces,
    several large deviations approaches to
    distributions of periodic points are known, see e.g., \cite{Kif94,PolSha96}.
    A natural problem to consider is an extension of these results to dynamical systems of non-compact spaces,
    notably to countable Markov shifts.
  One existing result closely related to Theorem \ref{distribution} is that of Fiebig et al. \cite{FieFieYur02}
  to which we shall come back later on.

 \subsection{Thermodynamic formalism for countable Markov shifts}
  Let $S$ be a countable set and $\mathbb N$ the set of non-negative integers.
  Let $S^{\mathbb N}$ denote
  the set of one-sided infinite sequences over $S$ endowed
with the product topology of the discrete topology on $S$, namely
 $$S^{\mathbb N}=\{x=(x_0,x_1,\ldots)\colon x_i\in S\quad\forall i\in\mathbb N\}.$$
This topology is metrizable with a metric 
 $d(x,y)=\exp\left({-\inf\{i\in\mathbb N\colon x_i\neq y_i\}}\right)$
 with the convention $\exp(-\infty)=0$.
 The left shift $\sigma$ acts continuously on $S^{\mathbb N}$ by $(\sigma x)_i=x_{i+1}$ $(\forall i\in\mathbb N)$.
 Let $T=(t_{ij})_{S\times S}$ be a matrix of zeros and ones with no column or row which is all made of zeros.
 {\it A} (one-sided) {\it topological Markov shift} $X$ generated by the transition matrix $T$ is given by
$$X=\{x\in S^{\mathbb N}\colon t_{x_ix_{i+1}}=1\quad\forall i\in\mathbb N\}.$$
  If $\#S=\infty$ (resp. $\#S<\infty$),  $X$ is called {\it a countable} (resp. {\it finite}) {\it Markov shift}.
  If all entries of the matrix are $1$, it is called {\it the full shift}.
  The restriction of the left shift to $X$ is denoted by $\sigma|_X$.

  For two strings $\omega=\omega_0\cdots \omega_{m-1}$, $\alpha=\alpha_0\cdots  \alpha_{n-1}$ of elements of $S$,
 let $\omega\alpha$ denote the concatenated string
$\omega_0\cdots \omega_{m-1}\alpha_0\cdots \alpha_{n-1}$.
An $n$-string  $\omega_0\cdots \omega_{n-1}$ 
  is {\it admissible} if $n=1$, or else $n\geq2$ and $t_{\omega_{i}\omega_{i+1}}=1$ holds for every $0\leq i\leq n-1$.
 Let $E^n$ denote the set of $n$-admissible strings and put $E^*=\bigcup_{n=1}^\infty E^n$.
   A countable Markov shift $X$ is {\it finitely irreducible}
if there exists a finite set $\Lambda\subset E^*$ 
such that for all $\omega,\omega'\in E^*$ there exists $\lambda\in\Lambda$
for which $\omega\lambda \omega'\in E^*$.
If $X$ is finitely irreducible and the finite set $\Lambda$
consists of strings of the same length $N$, then $X$ is called {\it finitely primitive}.
Notice that the set $\Lambda$ associated either with a finitely irreducible or 
primitive $X$ can be taken to be empty for the full shift (in which case $N=0$).

Let $\mathcal M$ denote the space of Borel probability measures on $X$
 endowed with the weak*-topology, and $\mathcal M(\sigma|_X)$ the set of shift-invariant elements 
of $\mathcal M$.
 The space $\mathcal M$ is metrizable with the bounded Lipschitz metric.
The Kolmogorov-Sina{\u\i} entropy of
each measure $\mu\in\mathcal M(\sigma|_X)$ with respect to $\sigma|_X$ is denoted by $h(\mu)$.
For a measurable function $\phi\colon X\to\mathbb R$ with $\sup\phi<\infty$ define
$$\mathcal M_\phi(\sigma|_X)=\left\{\mu\in\mathcal M(\sigma|_X)\colon\int\phi d\mu>-\infty\right\}.$$

 For each $n$-string $\omega=\omega_0\cdots \omega_{n-1}\in E^n$ 
define {\it an $n$-cylinder} 
 $$[\omega]=[\omega_0,\ldots,\omega_{n-1}]=\{x\in  X\colon x_i=\omega_i
 \quad0\leq \forall i\leq n-1\}.$$
 Put
  $$Z_n(\phi)=\sum_{\omega\in E^n}\sup_{[\omega]}\exp S_n\phi,$$
 where $S_n\phi=\sum_{i=0}^{n-1}\phi\circ \sigma^i$.
 Since $n\mapsto \log Z_n(\phi)$ is sub-additive, the limit  
 \begin{equation}\label{press}
 \lim_{n\to\infty}\frac{1}{n}\log Z_n(\phi).\end{equation}
in \eqref{press} exists, and in fact never $-\infty$.
 Define
 \begin{equation}\label{press0}
 P(\phi)=\sup\left\{h(\mu)+\int\phi d\mu\colon \mu\in\mathcal M_\phi(\sigma|_X)\right\}.\end{equation}
  We say {\it the variational principle} holds if 
  the limit in \eqref{press} is equal to $P(\phi)$.
If the variational principle holds and $P(\phi)<\infty$, then measures in $\mathcal M_\phi(\sigma|_X)$ which attain the supremum 
in \eqref{press0}
 are called {\it equilibrium states for the potential $\phi$}.

  A Borel probability measure $\mu_\phi$ on $X$ is {\it a Gibbs state (in the sense of Bowen) for the potential $\phi$}
(cf. \cite{Bow75,MauUrb03,Rue78,Sar99})
if there exist constants $c\geq1$ and $P\in\mathbb R$ such that
 for every $n\geq1$ and every $x\in X$, 
 \begin{equation}\label{Gibbs}
 c^{-1}\leq\frac{\mu_\phi  [x_0,\ldots,x_{n-1}] }{\exp\left(-Pn+S_n \phi(x)\right)}\leq c.
 \end{equation}
 If $\mu_\phi$ is shift-invariant, then it is called {\it a shift-invariant Gibbs state}.



  \subsection{Uniqueness of minimizer}
  Let $X$ be a countable Markov shift.
   A function $\phi\colon X\to\mathbb R$ is {\it summable} if 
$$Z_1(\phi)<\infty.$$
The summability of $\phi$ implies $\sup\phi<\infty$ and $\inf\phi=-\infty$. 
For a summable function $\phi$, define 
$$\beta_\infty(\phi)=\inf\{\beta\in\mathbb R\colon \text{$\beta\phi$ is summable}\}.$$
If  $\beta_0\geq0$ and $\beta_0\phi$ is summable then so is $\beta\phi$ for every $\beta>\beta_0$.
Hence, $0\leq\beta_\infty(\phi)\leq1$ holds.
A function $\phi\colon X\to\mathbb R$ has {\it summable variations} if 
 $$\sum_{n=1}^\infty\sup_{\omega\in E^n}\sup_{x,y\in [\omega]}\phi(x)-\phi(y)<\infty.$$
 If $X$ is finitely irreducible and $\phi$ is summable with summable variations,
 then the variational principle holds, and
there exists a unique shift-invariant Gibbs state $\mu_\phi$ for the potential $\phi$.
If moreover $\int\phi d\mu_\phi>-\infty$, then $\mu_\phi$ is the unique equilibrium state for the potential $\phi$
(see \cite{MauUrb03}).
One can verify this integrability provided $\beta_\infty(\phi)<1$. In fact, this condition also allows us 
to show the uniqueness of minimizer.
We obtain the following result.

    \begin{theorem}\label{minimizer-unique}
  Let $X$ be a finitely irreducible countable Markov shift and
 $\phi\colon X\to\mathbb R$ a summable function with summable variations satisfying
 $\beta_\infty(\phi)<1.$ 
Then,  the minimizer of the rate function is unique.
It coincides with 
the unique shift-invariant Gibbs state, and
 the unique equilibrium state for the potential $\phi$.
 \end{theorem}
 
 Theorem  \ref{minimizer-unique} has its analogue in statistical mechanics.
 Dobrushin-Lanford-Ruelle's variational principle \cite{Dob68a, Dob69,LanRue69} states that
 the set of shift-invariant Gibbs states defined by the DLR-equation, that of equilibrium states (measures minimizing the free energy) and 
 that of minimizers (of the relative entropy density) coincide
for shift-invariant absolutely summable interactions. 

  \begin{remark}
The following example inspired by \cite{ArnAve82,MauUrb03,Sar03} shows that the assumption 
$\beta_\infty(\phi)<1$ in Theorem \ref{minimizer-unique} is not removable.
Consider the potential $\phi\colon\mathbb N^{\mathbb N}\to\mathbb R$ given by $\phi(x)=\log p_{x_0}$
where $\{p_k\}_{k\in\mathbb N}$ is a sequence with $p_k\in(0,1)$, $\sum_{k\in\mathbb N} p_k=1$ 
and $\sum_{k\in\mathbb N} p_k\log p_k=-\infty$, e.g.,
$p_k\propto 1/(k(\log k)^2)$. 
The Gibbs state $\mu_\phi$ for the potential $\phi$ 
is the Bernoulli measure associated with 
the infinite probability vector $(p_k)_{k\in\mathbb N}$.
It is ergodic and hence a minimizer.
As $\int\phi d\mu_\phi=-\infty$, it is not an equilibrium state for the potential $\phi$.
\end{remark}

\subsection{Weighted equidistribution}
We now proceed to distributions of periodic points.
For each $n\geq1$ and $x\in X$
let $\delta_x^n$ denote the uniform probability distribution on $\{x,\sigma x,\ldots,\sigma^{n-1}x\}$.
Write ${\rm Per}_n(\sigma|_X)=\{x\in X\colon\sigma^nx=x\}.$
 \begin{theorem}\label{distribution}
   Let $X$ be a finitely primitive countable Markov shift and
 $\phi\colon X\to\mathbb R$ a summable function with summable variations satisfying
   $\beta_\infty(\phi)<1$. As $n\to\infty$,
 the Borel probability measure on $X$ defined by
$$\left( \sum_{x\in{\rm Per}_n(\sigma|_X)}\exp S_n\phi(x)\right)^{-1} \sum_{x\in{\rm Per}_n(\sigma|_X)}
\exp S_n\phi(x)\delta_x^n$$
  converges in the weak*-topology to the equilibrium state for the potential $\phi$.
\end{theorem}

For transitive countable Markov shifts,
  Fiebig et al. \cite{FieFieYur02} introduced the notion of {\it $Z$-recurrence}
 in terms of a pressure computed from weighted periodic points contained in one fixed $1$-cylinder.
  Under the $Z$-recurrence and some additional assumptions, they showed that
   the sequence of measures constructed 
 from these periodic points are tight, and all its weak*-accumulation points are 
 equilibrium states for the potential \cite[Theorem 4.2]{FieFieYur02}.
  Theorem \ref{distribution} does not follow from \cite[Theorem 4.2]{FieFieYur02}
  because it concerns all periodic points in $X$ and
  there are infinitely many $1$-cylinders.
The tightness of the sequence of measures in Theorem \ref{distribution} is a consequence of the exponential tightness \cite{DemZei98}
of the sequence $\{\eta_n\}_{n=1}^\infty$ of measures on $\mathcal M$ in Theorem \ref{Tak}. 

We apply Theorem \ref{distribution} to the Gauss map 
$$T\colon x\in(0,1]\mapsto  \frac{1}{x}-\left\lfloor\frac{1}{x}\right\rfloor\in[0,1).$$
Following the orbits of $T$ over the infinite Markov partition $\{(\frac{1}{k+1},\frac{1}{k}]\}_{k\geq1}$,
 one can model $T$ with the countable full shift. 
  The unbounded function
  $-\log|T'|$ 
  induces a potential $\phi$
  on the full shift which is summable with summable variations (see \cite[Section 7]{FieFieYur02}).
The Gauss measure $\frac{1}{\log2}\frac{dx}{1+x}$ corresponds to the unique equilibrium state
for the potential $\phi$ (see \cite[Theorem 22]{Wal82}), and
we have $\beta_\infty(\phi)=1/2$ (see \cite{May90}).
By Theorem \ref{minimizer-unique}, the equilibrium state for the potential $\phi$ is the unique minimizer
of the corresponding rate function.

For each integer $n\geq1$ and $x\in(0,1)\setminus\mathbb Q$,
let $\delta_x^n(T)$ denote the uniform probability distribution on $\{x,Tx,\ldots,T^{n-1}x\}$.
Write ${\rm Per}_n(T)=\{x\in(0,1)\setminus\mathbb Q\colon T^nx=x\}.$  
 The exponential instability of each point $x\in {\rm Per}_n(T)$ under the iteration of the Gauss map is quantified by its Lyapunov exponent $(1/n)\log|(T^n)'x|$.
It is well-known \cite{HarWri79,Khi64} that numbers in $\bigcup_{n=1}^\infty{\rm Per}_n(T)$ correspond to solutions of quadratic equations with integer coefficients, and 
their Lyapunov exponents are related to the rate of Diophantine approximation by the regular continued fraction expansion.

 \begin{prop}\label{equi-prop}
 As $n\to\infty$, the measure
$$\left(\sum_{x\in{\rm Per}_n(T)} |(T^n)'x|^{-1}\right)^{-1}\sum_{x\in{\rm Per}_n(T)}|(T^n)'x|^{-1}
\delta_x^n(T)$$
converges in the weak*-topology to the Gauss measure.
\end{prop}

\subsection{Finitely primitive uniformly expanding Markov maps}
Examples satisfying the key condition $\beta_\infty(\phi)<1$ in Theorem \ref{minimizer-unique} and Theorem \ref{distribution}
are in abundance. We show this by linking this condition to another on dimension of invariant sets
of Markov interval maps. 

   Let $f\colon\varDelta\to [0,1]$ satisfy the following:

 \begin{itemize}
 
\item
   there exists a family $\{\varDelta_a\}_{a\in \mathbb N}$ 
of intervals with disjoint interiors
such that $\varDelta=\bigcup_{a\in \mathbb N}\varDelta_a$
and $f|_{\varDelta_a}=f_a$ for each $a\in \mathbb N$.

 \item for each $a\in \mathbb N$, $f_a$ is a $C^1$ diffeomorphism onto its image
 with appropriate one-sided derivatives.

 \item if $a,b\in \mathbb N$ and $f\varDelta_a\cap \varDelta_b$ 
 has non-empty interior,
 then $f\varDelta_a\supset \varDelta_b$.
 
 \end{itemize}
 The maximal invariant set of $f$ is defined by
$$J= \bigcap_{n=0}^{\infty}f^{-n}\varDelta.$$ 
Let $\dim_HJ\in[0,1]$ denote the Hausdorff dimension of $J$.
We assume (some iterate of) $f$ is uniformly expanding, i.e.,
  there exists $\kappa>1$ such that
 $|f'|>\kappa$ on $\varDelta$. Then, $f$ is modeled by the countable Markov shift
 whose transition matrix is determined by the Markov partition $\{\varDelta_a\}_{a\in \mathbb N}$.
 We moreover assume the Markov shift is finitely primitive, and $-\log|f'|$ induces a summable function $\phi$ with summable variations. By
  \cite[Theorem 4.2.13]{MauUrb03} we have
$$\dim_HJ=\inf\{\beta\geq0\colon P(\beta\phi)<0\}.$$
From this and \cite[Proposition 2.1.9]{MauUrb03} we obtain $$\beta_\infty(\phi)\leq\dim_HJ.$$
Therefore,  the condition $\dim_HJ<1$ implies $\beta_\infty(\phi)<1$.
In particular, all periodic points of $f$ weighted with their Lyapunov exponents equidistribute as the periods tend to infinity 
as in Proposition \ref{equi-prop}.

\subsection{Bowen-Series maps for free Fuchsian groups}   
We give another application of Theorem \ref{distribution} to cuspidal windings for geodesic flows on hyperbolic surfaces. 
Our presentation below 
is essentially a compressed version of \cite[Section 2]{JaeTak}.
For more details, see there and the references therein.

Let $\mathbb D\subset \mathbb R^2$ denote the open unit disk around the origin
and $\mathbb S^1$ the boundary of $\mathbb D$.
 {\it A Fuchsian group} is a discrete subgroup of M\"obius transformations preserving $\mathbb D$.
Let $G$ be a finitely generated free Fuchsian group with parabolic elements.
We denote by $\Lambda(G)\subset \mathbb S^1$ the limit set of  $G$.
Fix a Dirichlet fundamental domain $R\subset\mathbb{D}$ for $G$.
Each side $s$ of $R$ gives rise to a side-pairing transformation
$g_{s}\in G$, and the set 
\[
G_{0}=\left\{ g_{s}\colon s\text{ }\text{side of }R\right\} \subset G
\]
is a symmetric set of generators of $G$. Since $G$ is free,  each vertex of $R$ 
is in $\mathbb S^1$ and it is a parabolic
fixed point of some element of $G$. For simplicity we assume 
the vertex of $R$ is unique, denoted by $p$,
and there exists $\gamma_0\in G_0$ such that $\gamma_0(p)=p$.  We define 
\[
\Gamma_0=\{ \gamma_0^{\pm 1}  \}\ \text{ and }\ H_{0}=G_{0}\setminus\Gamma_{0}.
\]
Each side $s$ of $R$ is contained  in the isometric
circle of $g_s$.
For $g\in G_0$ define
\[
\varDelta_{g}=\mathbb S^1\cap \left\{ z\in\mathbb{R}^{2}\colon |g'(z)|\ge1\right\},
\]
and put $\varDelta=\bigcup_{g\in G_{0}}\varDelta_{g}$.
Following \cite{BowSer79} we introduce the Bowen-Series
map $f:\varDelta\rightarrow\mathbb S^1$ by 
\[
f|_{\varDelta_{g}}=g|_{\varDelta_{g}},\quad g\in G_{0}.
\]

The map $f$ determines a one-to-one correspondence 
between  $\Lambda(G)$ and a subset 
 of the set  $\{(\omega_n)_{n=0}^\infty\in G_0^{\mathbb N}
 \colon \omega_{n}\omega_{n+1}\neq1\in G\ \forall n\geq0\}$ 
 of symbolic sequences.
Let $\Lambda_c(G)$ denote the conical limit set of $G$. 
 For $x\in \Lambda_c(G)$, 
 its symbolic sequence $\omega$ is decomposed into a  sequence of blocks $(B_{i}(x))_{i\ge1}$ as in \cite{JaeKesMun16}.
Each element of $H_0$ in $\omega$ forms a block of length one. For elements of $\Gamma_0$
in $\omega$ we build maximal blocks of consecutive appearances of the
same element. Either  $B_i(x)=h$ for some $h\in H_0$, or $B_i(x)=\gamma_0^n$ holds
for some $n\in\mathbb Z\setminus\{0\}$. 

This decomposition has the following geometric interpretation.
For a fixed initial value $o\in \mathbb D/G$, the set of closed geodesics $\gamma:[0,\infty)\rightarrow \mathbb D/G$ satisfying $\gamma(0)=o$ can be identified with 
(not necessarily prime) periodic orbits of $f$ in $\Lambda_c(G)$.
The  symbolic sequence $\omega$
 records the sides of $R$ crossed by $\gamma(t)$ consecutively as $t\rightarrow \infty$.
A block $B_i(x)=\gamma_0^{\pm n}$ of length $n\ge2$ means that the geodesic $\gamma$  spirals
$n-1$ times around the cusp $\gamma_0$. 



For each reduced word
$\omega=\omega_0\cdots\omega_{n-1}\in G_0^n$ of length $n\geq2$, let
$\varDelta_\omega$ denote the domain of the diffeomorphism $f|_{\varDelta_{\omega_{n-1}}}\circ\cdots\circ f|_{\varDelta_{\omega_0}}$.
We now define an induced map
 $\tilde{f}:\bigcup_{\omega\in F}\varDelta_{\omega}\rightarrow\mathbb{S}^1$
 by 
\[
\tilde{f}x=f^{|B_{1}(x)|}x,
\]
where  $|B_1(x)|$ denotes the word length of the block $B_1(x)$, and 
$F$ is the countable set of reduced words in $\bigcup_{n=2}^\infty G_{0}^n$  given by
\[
F=\bigcup_{n=1}^\infty\left\{ \gamma^{n}g\colon
\gamma\in\Gamma_{0},g\in G_{0}\setminus\left\{ \gamma^{\pm1}\right\} \right\} \cup\left\{ hg\colon h\in H_{0},\,\,g\in G_{0}\setminus\left\{ h^{-1}\right\} \right\} .
\]
The maximal invariant set of $\tilde{f}$ is $\Lambda_{c}(G)$, and
 $B_i(x)=B_1(\tilde f^{i-1} x)$ holds for all $x\in\Lambda_c(G)$ and $i\geq1$.
 The $\tilde f$ determines a transition matrix and an associated countable Markov shift
$X$ with symbols in $F$. Since $\tilde f$ is uniformly expanding, 
there is a conjugacy $\pi\colon X\to \Lambda_c(G)$ with $\tilde f\circ\pi=\pi\circ\sigma|_X$.
Moreover, $\tilde f$-invariant Borel probability measures and shift-invariant ones are in one-to-one correspondence.
For each $n\geq1$ and 
 $x\in \Lambda_c(G)$ let $\delta_x^n(\tilde f)$ denote the uniform probability distribution 
 on $\{x,\tilde f x,\ldots,\tilde f^{n-1}x\}$.
Write
${\rm Per}_n(\tilde f)=\{x\in\Lambda_c(G)\colon B_i(x)=B_{i+n}(x)\ \forall i\geq1\}.$
For each prime periodic point $x\in {\rm Per}_n(\tilde f)$ of $f$, the total number of cuspidal windings of
the corresponding closed geodesic is $-n+\sum_{i=1}^n|B_i(x)|$.

 \begin{prop} Assume $f$ is topologically mixing. 
As $n\to\infty$, the measure 
$$\left(\sum_{x\in{\rm Per}_n(\tilde f)}\exp\left(-\sum_{i=1}^n|B_i(x)|\right)\right)^{-1}\sum_{x\in{\rm Per}_n(\tilde f)}\exp\left(-\sum_{i=1}^n|B_i(x)|\right)
\delta_x^n(\tilde f)$$
converges in the weak*-topology to the Borel probability measure on $\Lambda_c(G)$ 
corresponding to the equilibrium state for the potential $-|B_1\circ\pi|$.
\end{prop}



\subsection{LDPs for countable Markov shifts}
We now give a precise formulation of the LDP in our context. 
Let $X$ be topological Markov shift and $\phi\colon X\to\mathbb R$ a measurable function such that $\sup\phi<\infty$
and $P(\phi)<\infty$.
The (minus of the) free energy $F \colon \mathcal M \to [-\infty, 0]$ is defined by
\begin{equation*}\label{ef}
\begin{split}F(\nu)
=
\begin{cases}-P(\phi)+h(\nu)+\int\phi d\nu &\text{ if $\nu\in\mathcal M_\phi(\sigma|_X)$};\\ 
-\infty&\text{ otherwise.}\end{cases}
\end{split}
\end{equation*}
In the case $X$ is a topologically mixing finite Markov shift and $\phi$ is a H\"older continuous potential,
the LDP for empirical means was established in \cite{Kif90,Tak84}, and for 
weighted periodic points in \cite{Kif94}.  In these settings the rate function is $-F$, and therefore
 the minimizer coincides with the Gibbs-equilibrium state for the potential $\phi$.
The same characterization of the rate function is no longer true 
for countable Markov shifts, as we explain below.

 Let $X$ be a finitely irreducible countable Markov shift and
 $\phi\colon X\to\mathbb R$ a summable function with summable variations.
 Several level-2 LDPs have been established in \cite{Tak18}.
Since the entropy is not upper semi-continuous \cite[p.774]{JMU14}, 
  $-F$ cannot be the rate function. The rate function $I\colon\mathcal M\to[0,\infty]$ in \cite{Tak18} is given by
 \begin{equation}\label{rf}I(\mu)=-\inf_{\mathcal G\ni\mu}\sup_\mathcal G F,\end{equation}
 where the infimum is taken over all open subsets $\mathcal G$ of $\mathcal M$ containing $\mu$.

 Let
 $\mu_\phi$ denote the shift-invariant Gibbs state for the potential $\phi$.
We introduce two sequences $\{\xi_n\}_{n=1}^\infty$, $\{\eta_n\}_{n=1}^\infty$ of 
Borel probability measures on $\mathcal M$: 
  
  \begin{itemize}
 
\item[]\noindent{1. (Empirical means)}  
for each $n\geq1$ and 
 $x\in X$ let $\delta_x^n$ denote the uniform probability distribution 
 on $\{x,\sigma x,\ldots,\sigma^{n-1}x\}$.
Then $\xi_n$ is the distribution of the random variable 
 $x\mapsto\delta_x^n$ on the probability space $(X,\mu_\phi)$;

\item[]\noindent{2. (Weighted periodic points)} 
for each integer $n\geq1$ define
$$\eta_n=\left(\sum_{x\in {\rm Per}_n(\sigma|_X)}\exp S_n\phi(x)\right)^{-1}
\sum_{x\in{\rm Per}_n(\sigma|_X)} \exp S_n\phi(x)\delta_{\delta_x^n},$$
where
 $\delta_{\delta_x^n}$ denotes the unit point mass at $\delta_x^n$.

\end{itemize}

 \begin{theorem}{\rm (\cite[Theorem A]{Tak18})}\label{Tak}
    Let $X$ be a finitely irreducible countable Markov shift and
 $\phi\colon X\to\mathbb R$ a summable function with summable variations.
  Then $\{\xi_n\}_{n=1}^\infty$ satisfies the LDP with 
the good rate function $I$.
 If $X$ is finitely primitive, then 
   $\{\eta_n\}_{n=1}^\infty$ satisfies the LDP
  with the same good rate function. \end{theorem}

From Theorem \ref{minimizer-unique} and Theorem \ref{Tak} we obtain an exponential decay.

        \begin{cor}\label{expdecay}
         Let $X$, $\phi$ be as in Theorem \ref{minimizer-unique}. 
        Then, for any weak*-open set $\mathcal G$ containing $\mu_\phi$, one has
        $$\varlimsup_{n\to\infty}\frac{1}{n}\log\mu_\phi\{x\in X\colon\delta_x^n\notin\mathcal G\}\leq-\inf_{\mathcal M\setminus\mathcal G} I<0.$$
        If moreover $X$ is finitely primitive, then
                $$\varlimsup_{n\to\infty}\frac{1}{n}\log
        \frac{\sum_{\stackrel{x\in{\rm Per}_n(\sigma|_X)}{\delta_x^n\notin\mathcal G}} \exp{S_n\phi(x)}}{\sum_{x\in{\rm Per}_n(\sigma|_X)} \exp{S_n\phi(x)}}\leq-\inf_{\mathcal M\setminus\mathcal G} I<0.$$

                \end{cor}

\section{Proofs of the main results}
 A proof of Theorem \ref{minimizer-unique} amounts to showing the next theorem which 
identifies minimizers as equilibrium states.
 
  \begin{theorem}\label{uniquemin}
  Let $X$ be a finitely irreducible countable Markov shift and
 $\phi\colon X\to\mathbb R$ a uniformly continuous summable function
 satisfying $\beta_\infty(\phi)<1$.  Assume there exists a Gibbs state for the potential $\phi$.
 Then, any minimizer of the rate function 
 is an equilibrium state for the potential $\phi$.
 \end{theorem}
 We finish the proof of Theorem \ref{minimizer-unique} and Theorem \ref{distribution} assuming Theorem \ref{uniquemin}.
 
 \subsection*{Proof of Theorem \ref{minimizer-unique}}
   Let $X$ be finitely irreducible and
 $\phi\colon X\to\mathbb R$ be summable with summable variations satisfying $\beta_\infty(\phi)<1$.
By \cite[Theorem 1.5]{MauUrb03} the variational principle holds, and by \cite[Theorem 2.2.4]{MauUrb03}
there exists a unique shift-invariant Gibbs state for the potential $\phi$ which we denote by $\mu_\phi$. 
     
  Fix $\delta\in(0,1-\beta_\infty(\phi))$.
  Since $\phi$ has summable variations, there exists a constant $C(\phi,\delta)>0$
 such that 
$\sup_{[k]}|\phi|\leq C(\phi,\delta)\sup_{[k]}\exp(-\delta \phi)$ for every $k\in\mathbb N$.
  Using this and \eqref{Gibbs} we obtain 
  \begin{align*}
  \sum_{k\in\mathbb N}\sup_{[k]}|\phi|\mu_\phi[k]
  &\leq ce^{-P(\phi)}\sum_{k\in\mathbb N} \sup_{[k]}\exp\phi\sup_{[k]}|\phi|\\
  &\leq
    ce^{-P(\phi)}C(\phi,\delta)
 Z_1\left((1-\delta)\phi\right)<\infty,
  \end{align*}
  and thus $\int\phi d\mu_\phi>-\infty$.
  By \cite[Theorem 2.2.9]{MauUrb03}, $\mu_\phi$ is the unique equilibrium state for the potential $\phi$.
From this and Theorem \ref{uniquemin}, the minimizer is unique and it is $\mu_\phi$.
  \qed
  
  \subsection*{Proof of Theorem \ref{distribution}}
   Let $X$ be finitely primitive and
 $\phi\colon X\to\mathbb R$ be summable with summable variations satisfying
   $\beta_\infty(\phi)<1$. 
Let $p_n$ denote the Borel probability measure on $X$ in the statement of Theorem \ref{distribution}.
For an arbitrary bounded continuous function $\varphi\colon X\to\mathbb R$,
  a direct calculation gives
\begin{align*}\int\varphi dp_{n}&=\frac{1}{n}
\sum_{i=0}^{n-1}\int\varphi d\left(
p_{n}\circ(\sigma|_X)^{-i}\right)\\
&=\frac{1}{n}\left(\sum_{x\in {\rm Per}_{n}(\sigma|_X)}\exp S_{n}\phi(x)\right)^{-1}
\sum_{x\in{\rm Per}_{n}(\sigma|_X)} \exp S_{n}\phi(x)\sum_{i=0}^{n-1}\varphi(\sigma^ix)\\
&=\left(\sum_{x\in {\rm Per}_{n}(\sigma|_X)}\exp S_{n}\phi(x)\right)^{-1}
\sum_{x\in{\rm Per}_{n}(\sigma|_X)} \exp S_{n}\phi(x)\int\varphi d\delta_x^{n}\\
&=\int\left(\int \varphi d\mu \right)d\eta_{n}(\mu).
\end{align*}
By Theorem \ref{Tak}, the sequences
 $\{\eta_n\}_{n=1}^\infty$ satisfy the LDP
 with the good rate function $I$.
 Since the rate function is good and the minimizer is unique,
 $\{\eta_n\}_{n=1}^\infty$ converges in the weak*-topology to the unit point mass at the unique minimizer $\mu_\phi$.
Since
the functional $\mu\in\mathcal M\mapsto\int\varphi d\mu$ is bounded continuous, 
$\lim_{n\to\infty}\int\left(\int \varphi d\mu \right)d\eta_{n}(\mu)=\int\varphi d\mu_\phi$ holds, and thus
$\lim_{n\to\infty}\int\varphi dp_{n} =\int\varphi d\mu_\phi$
as required. \qed

\subsection{Preliminaries}
The rest of this paper is dedicated to a proof of Theorem \ref{uniquemin}.
If not stated otherwise, $X$ denotes a general countable Markov shift.

\begin{lemma}\label{modify1}
Let $\phi\colon X\to\mathbb R$ be a summable 
function.
For any $\delta>0$ there exists a constant $K(\delta)\in\mathbb R$ such that if $\mu\in\mathcal M_\phi(\sigma|_X)$
satisfies $-h(\mu)/\int\phi d\mu>\beta_\infty(\phi)+\delta$ then $ \int\phi d\mu\geq K(\delta)$.
\end{lemma}

\begin{proof}
For $\delta>0$ put $\beta_0=\beta_\infty(\phi)+\delta/2$.
 Then
$$h(\mu)+\beta_0\int\phi d\mu\leq P(\beta_0\phi)\leq \lim_{n\to\infty}\frac{1}{n}\log Z_n(\beta_0\phi)<\infty.$$
The second inequality follows from \cite[Theorem 2.1.7]{MauUrb03}, and the last one follows from
the summability gives of $\beta_0\phi$.
Put
$K(\delta)=\frac{P(\beta_0\phi)}{\beta_0-\beta_\infty(\phi)-\delta}$.
The first inequality and the assumption on $\mu$ give
$$(\beta_0-\beta_\infty(\phi)-\delta)\int\phi d\mu\leq P(\beta_0\phi),$$
and therefore
$\int\phi d\mu\geq K(\delta)$ as required.
\end{proof}


\begin{lemma}\label{div}
Let $\phi\colon X\to\mathbb R$ be a summable function
satisfying $\beta_\infty(\phi)<1.$
Let $\{\mu_k\}_{k=1}^\infty$ be a sequence in $\mathcal M_\phi(\sigma|_X)$ such that 
$\{F(\mu_k)\}_{k=1}^\infty$ converges to a finite number as $k\to\infty$. Then
$\inf_{k\geq1}\int\phi d\mu_k>-\infty.$
\end{lemma}
\begin{proof}
If the infimum is $-\infty$, then
 it is possible to take a subsequence $\{\mu_{k_i}\}_{i=1}^\infty$ of $\{\mu_k\}_{k=1}^\infty$ such that
$\int\phi d\mu_{k_i}\to-\infty$ and
$-h(\mu_{k_i})/\int\phi d\mu_{k_i}\to1$ as $i\to\infty$. 
Fix $\delta\in(0,1-\beta_\infty(\phi))$. 
Since $-h(\mu_{k_i})/\int\phi d\mu_{k_i}>\beta_\infty(\phi)+\delta$ holds for sufficiently large $i$, we 
obtain a contradiction to Lemma \ref{modify1}.
\end{proof}

To prove Theorem \ref{uniquemin}, we take a sequence of measures converging to a minimizer,
and evaluate the free energies along this sequence. The main difficulty is the lack of upper semi-continuity of entropy. 
To overcome this we show a limited form of upper semi-continuity in the following form.

\begin{theorem}\label{ups}
Let $\phi\colon X\to\mathbb R$ be a uniformly continuous summable function satisfying
 $\beta_\infty(\phi)<1$.
 Let $\{\mu_j\}_{j=1}^\infty$ be a sequence in $\mathcal M_\phi(\sigma|_X)$ which converges to $\mu_0\in\mathcal M_\phi(\sigma|_X)$
 in the weak*-topology as $j\to\infty$. Assume
$P(\phi)-\int\phi d\mu_0>0$. 
If  \begin{equation}\label{assu}
 \varliminf_{j\to\infty}\frac{h(\mu_j)}{P(\phi)-\int\phi d\mu_j}>\beta_\infty(\phi),\end{equation}
 then
\begin{equation}\label{desired}
\frac{h(\mu_0)}{P(\phi)-\int\phi d\mu_0}\geq\varlimsup_{j\to\infty}\frac{h(\mu_j)}{P(\phi)-\int\phi d\mu_j}.\end{equation}
\end{theorem}

\begin{remark}\label{first}
As continuous functions on $X$ bounded from above are integrated 
upper semi-continuously,
the assumption $P(\phi)-\int\phi d\mu_0>0$ in Theorem \ref{ups} implies
 $P(\phi)-\int\phi d\mu_j>0$ for sufficiently large $j$.
\end{remark}

A main inspiration for Theorem \ref{ups} is from
\cite[Lemma 6.5]{FanJorLiaRam16}, which is used for refined descriptions of multifractal spectra
 of fully branched Markov interval maps. 
Below we finish the proof of Theorem \ref{uniquemin} assuming Theorem \ref{ups}.

\subsection*{Proof of Theorem \ref{uniquemin}}
  Let $\mu_*\in\mathcal M(\sigma|_X)$ be a minimizer. Let
 $\{\mu_k\}_{k=1}^\infty$ be a sequence in $\mathcal M_\phi(\sigma|_X)$ 
which converges in the weak*-topology to $\mu_*$ 
with $\lim_{k\to\infty}F(\mu_k)=0$. 
Lemma \ref{div} gives 
$\inf_{k\geq1}\int\phi d\mu_k>-\infty$; 
As $\phi$ is continuous bounded from above, it is integrated upper semi-continuously
 \cite[Lemma 1]{JMU14}.
In particular, $\int\phi d\mu_*>-\infty$ holds.
    If $\varliminf_{k\to\infty}h(\mu_k)=0$, then
 for a subsequence $\{\mu_{k_j}\}_{j=1}^\infty$ with $\lim_{j\to\infty}h(\mu_{k_j})=0$
 the formula \eqref{rf} gives
\begin{align*}
0=\lim_{j\to\infty}F(\mu_{k_j})&\leq \varlimsup_{j\to\infty} \int\phi d\mu_{k_j}\leq h(\mu_*)+\int\phi d\mu_*,
\end{align*}
namely $\mu_*$ is an equilibrium state for the potential $\phi$.
If $\varliminf_{k\to\infty}h(\mu_k)>0$, then
$\varliminf_{k\to\infty}(-\int\phi d\mu_k)>0$ and 
\begin{equation*}
0=\lim_{k\to\infty}F(\mu_k)=\lim_{k\to\infty}
\left(-\int\phi d\mu_k\right)\left(\frac{h(\mu_k)}{-\int\phi d\mu_k}-1\right).
\end{equation*}
It follows that
$$\lim_{k\to\infty}\left(\frac{h(\mu_k)}{-\int\phi d\mu_k}-1\right)=0.$$
 We have $-\int\phi d\mu_*\geq h(\mu_*)\geq 0$.
If $-\int\phi d\mu_*=0$, then clearly
$\mu_*$ is an equilibrium state for the potential $\phi$. 
If $-\int\phi d\mu_*>0$, then Theorem \ref{ups} gives
$$\frac{h(\mu_*)}{-\int\phi d\mu_*}-1\geq0,$$
namely $\mu_*$ is an equilibrium state for the potential $\phi$.
\qed

\subsection*{Proof of Theorem \ref{ups}}\label{limit}
We shall view measures in $\mathcal M(\sigma|_X)$ as measures 
on the full shift $\mathbb N^{\mathbb N}$, and project them
to the canonical finite subsystems $\Sigma_p$ $(p\in\mathbb N)$,
show for each $p\in\mathbb N$ a $p$-th approximation of the inequality \eqref{desired}, and finally let $p\to\infty$ to obtain \eqref{desired}.
The condition $\beta_\infty(\phi)<1$  will be used to control tails arising in approximations of entropy and integrals of the potential.

Considering $\phi-P(\phi)$ instead of $\phi$, we may assume $P(\phi)=0$.
Put $\beta_\infty:=\beta_\infty(\phi)<1$.
Let $\{\mu_j\}_{j=1}^\infty$ be a sequence in $\mathcal M_\phi(\sigma|_X)$ which converges to $\mu_0\in \mathcal M_\phi(\sigma|_X)$ and assume $-\int\phi d\mu_0>0$.
Since $\beta_\infty\geq0$ and
$\varliminf_{j\to\infty}(-\int\phi d\mu_j)>0$ as in Remark \ref{first},
\eqref{assu} implies
 \begin{equation}\label{zwei}\varliminf_{j\to\infty}h(\mu_j)>0.\end{equation} 
  It is convenient to split the rest of the proof into two cases.
\medskip


\noindent{\it Case 1: $\phi$ is constant on each $1$-cylinder of $X$.}
We view each measure $\mu\in\mathcal M(\sigma|_X)$ as an element of $\mathcal M(\sigma)$
by setting $\mu(A)=\mu(A\cap X)$ for any Borel subset $A$ of $\mathbb N^{\mathbb N}$.
This extension preserves entropy.
For each $k\in\mathbb N$ let $\langle k \rangle$ denote
 the corresponding $1$-cylinder of $\mathbb N^{\mathbb N}$.
For each $p\in\mathbb N$ we consider the shift-invariant subspace
$$\Sigma_p=\{x\in \mathbb N^{\mathbb N}\colon x_i\leq p\quad\forall i\in\mathbb N\}.$$
Define a projection $\pi_p\colon\mathbb N^{\mathbb N}\to\Sigma_p$
as follows: for each $x=x_0x_1x_2\cdots\in \mathbb N^{\mathbb N}$ define $\pi_p(x)\in\Sigma_p$ by
 replacing in the sequence 
$x_0x_1x_2\cdots$ all symbols greater than or equal to $p+1$ by the symbol $p$. 
For each $\mu\in\mathcal M(\sigma)$, write $\mu|_p$ for $\mu\circ\pi_p^{-1}$.
Since $\pi_p$ commutes with the shift, $\mu|_p$ is a $\sigma|_{\Sigma_p}$-invariant measure.
Put
$$c_p(\mu)=\sum_{k=p+1}^\infty\mu[k]
 \quad\text{and}\quad K_p(\mu)=-\sum_{k=p+1}^\infty\phi(k)\mu[k].$$
Notice that $\int\phi d\mu>-\infty$ if and only if $K_p(\mu)\to0$ as $p\to\infty$.

\begin{prop}\label{approx-full}
 If
  $\phi\colon X\to\mathbb R$ is constant on each $1$-cylinder of $X$, then
  for any $\delta>0$ there exists $p_0\geq0$ such that for every
   $\mu\in\mathcal M_\phi(\sigma|_X)$ and every
   $p\geq p_0$,
\begin{align*}h(\mu)-h(\mu|_p)\leq&  -(1-c_p(\mu))\log(1-c_p(\mu))\\
&-c_p(\mu)\log c_p(\mu)
+(\beta_\infty+\delta)K_p(\mu).\end{align*}
\end{prop}
\begin{proof}
Before proceeding let us summarize basic facts on entropy. 
Let $\mathscr A=\{A_k\}_{k\in\mathbb N}$ be a countable partition of $X$ into Borel sets
and let $\mu\in\mathcal M(\sigma|_X)$.
The entropy of $\mathscr A$ with respect to $\mu$ is the number
$$H_\mu(\mathscr A)=-\sum_{k\in\mathbb N}\mu(A_k)\log\mu(A_k),$$
with the convention $0\log0=0$.
If $H_\mu(\mathscr A)<\infty$ then define
$$h_\mu(\mathscr A)=\lim_{n\to\infty}\frac{1}{n}H_\mu\left(\bigvee_{i=0}^{n-1}\sigma^{-i}\mathscr A\right).$$
where the symbols $\bigvee$ and $\vee$ denote the join of the partitions $\sigma^{-i}\mathscr{A}$ $(0\leq i\leq n-1)$.
Since $n\mapsto H_\mu\left(\bigvee_{i=0}^{n-1}\sigma^{-i}\mathscr A\right)$ is sub-additive,
this limit exists, is finite and $h_\mu(\mathscr A)\leq H_\mu(\mathscr A)$.
We have $h(\mu)=\sup_{\mathscr A} h_\mu(\mathscr A)$ where the supremum is taken over all countable 
partitions $\mathscr A$ of $X$ with $H_{\mu}(\mathscr A)<\infty$. 
If $\mathscr A$ is a generator of the Borel sigma-field of $X$ and $H_{\mu}(\mathscr A)<\infty$ 
then $h(\mu)= h_{\mu}(\mathscr A)$.

Now, consider two partitions of $X$:
$$
\mathscr A_p=\{X\cap\pi_p^{-1}\langle 0\rangle,\ldots,X\cap\pi_p^{-1}\langle p\rangle\},\ \mathscr B_p=\left\{\bigcup_{k=0}^p[k],[p+1],[p+2],\ldots\right\}.$$
It is easy to check that $\mathscr A_p\vee\mathscr B_p$
  are generators
 of the Borel sigma-field of $X$.
 \begin{lemma}\label{lemma2}
 For every $\mu\in\mathcal M(\sigma)$ with finite entropy, $H_\mu(\mathscr B_p)<\infty$.
Moreover,  $h_{\mu}(\mathscr A_p\vee\mathscr B_p)\leq h_{\mu}(\mathscr A_p)+h_{\mu}(\mathscr B_p).$ \end{lemma}
 \begin{proof}
  The first assertion follows from \cite[Lemma 2.1]{Tak19}.
  For each integer $n\geq1$ we introduce a finite partition
$$  \mathscr B_{p,n}=\left\{\bigcup_{k=0}^p[k],[p+1],\ldots,[p+n],\bigcup_{k=p+n+1}^\infty[k]\right\}.$$
Then $h_{\mu}(\mathscr A_p\vee\mathscr B_{p,n})\leq h_{\mu}(\mathscr A_p)+h_{\mu}(\mathscr B_{p,n})
$ holds.
The $\{\mathscr B_{p,n}\}_{n=1}^\infty$ defines an increasing sequence of sub Borel sigma-fields of $X$
satisfying $\bigvee_{n=1}^\infty\mathscr B_{p,n}=\mathscr B_p.$
By \cite[Theorem 4.22]{Wal82},
$\lim_{n\to\infty}h_\mu(\mathscr B_{p,n})= h_\mu(\mathscr B_p)$,
and similarly $\lim_{n\to\infty}h_\mu(\mathscr A_p\vee\mathscr B_{p,n})= h_\mu(\mathscr A_p\vee\mathscr B_{p})$.
 \end{proof}
 
 \begin{lemma}\label{lemma3}
 For every $\mu\in\mathcal M(\sigma|_X)$,
 $h_{\mu}(\mathscr A_p)=h(\mu|_p).$
\end{lemma}
\begin{proof}
Consider finite partitions
$\mathscr C_p=\{\langle 0\rangle\cap\Sigma_p,\ldots,\langle p\rangle\cap\Sigma_p\}$ of $\Sigma_p$, and 
$\pi_p^{-1}\mathscr C_p=
\{\pi_p^{-1}\langle 0\rangle,\ldots,\pi_p^{-1}\langle p\rangle\}$
of  $\mathbb N^{\mathbb N}$.
Then $h_{\mu}(\mathscr A_p)=h_{\mu}(\pi_p^{-1}\mathscr C_p)$ holds.
 Since $\pi_p$ commutes with the shift,
  $h_{\mu}(\pi_p^{-1}\mathscr C_p)=h_{\mu|_p}(\mathscr C_p)$. 
 Since $\mathscr C_p$ is a generator of the Borel sigma-field of $\Sigma_p$,
  $h_{\mu|_p}(\mathscr C_p)=h(\mu|_p)$ holds.
\end{proof}

Returning to the proof of Proposition \ref{approx-full},
let $\mu\in\mathcal M_\phi(\sigma|_X)$. Then $h(\mu)<\infty$ holds.
In the case $c_p(\mu)=0$ there is nothing to prove since
$h(\mu)-h(\mu|_p)=0$.
Hence we assume $c_p(\mu)>0$.
 Lemma \ref{lemma2} gives
$h(\mu)=h_{\mu}(\mathscr A_p\vee\mathscr B_p)\leq h_{\mu}(\mathscr A_p)+h_{\mu}(\mathscr B_p).$
Using this and Lemma \ref{lemma3} we have
\begin{equation}\label{sav}
\begin{split}
h(\mu)-h(\mu|_p)&\leq h_\mu(\mathscr B_p)\leq H_{\mu}(\mathscr B_p)\\
&=-\sum_{k=0}^p\mu[k]\log\sum_{k=0}^p\mu[k]-\sum_{k=p+1}^\infty\mu[k]\log\mu[k].
\end{split}
\end{equation}
To treat the last summand in \eqref{sav},
define a potential $\varphi\colon\mathbb N^{\mathbb N}\to\mathbb R$
which is constant on each $1$-cylinder of $\mathbb N^{\mathbb N}$ by $\varphi|_{\langle k\rangle}=\phi(k)$.
The summability of $\phi$ implies that of $\varphi$.
 Denote by $\nu_p$ the Bernoulli measure on $\mathbb N^{\mathbb N}$
which assigns to each $1$-cylinder $\langle k\rangle$, $k\geq p+1$ the probability
$\mu[k]/c_p(\mu)$. Notice that
\begin{equation}\label{entro1}
h(\nu_p)=-\sum_{k= p+1}^\infty\frac{\mu[k]}{c_p(\mu)}\log \frac{\mu[k]}{c_p(\mu)}
\ \text{ and }\ \int\varphi d\nu_p=-\frac{K_p(\mu)}{c_p(\mu)}>-\infty.\end{equation}
 Since
$\int\varphi d\nu_{p}\leq \sup_{k\geq p+1}\phi(k),$
the summability of $\phi$ implies $\int\varphi d\nu_{p}\to-\infty$  as $p\to\infty$.
 Lemma \ref{modify1} applied to $(\mathbb N^{\mathbb N}, \varphi)$ shows that
  for any $\delta>0$ there exists $p_0\geq1$ independent of $\mu$ such that
 for every $p\geq p_0$,
$-h(\nu_p)/\int\varphi d\nu_{p}\leq  \beta_\infty(\varphi)+
\delta.$
Since $\phi$ is constant on each $1$-cylinder, $\beta_\infty(\varphi)=\beta_\infty$ holds, and
 we obtain
\begin{equation}\label{entro2}
c_{p}(\mu)h(\nu_p)\leq K_p(\mu)( \beta_\infty+
\delta).\end{equation}
Plugging \eqref{entro1} into \eqref{entro2} and then rearranging the result gives
$$-\sum_{k=p+1}^\infty \mu[k]\log\mu[k]\leq -c_p(\mu)\log c_p(\mu)
+K_p(\mu)( \beta_\infty+
\delta).$$
Plugging this inequality into \eqref{sav} yields the desired one.
\end{proof}

In view of \eqref{assu} fix $\delta>0$ and then $\gamma_0\in(0,1)$ such that
\begin{equation}\label{view}
\gamma_0\cdot\varliminf_{j\to\infty}\frac{h(\mu_j)}{-\int\phi d\mu_j}\geq\beta_\infty+\delta.\end{equation}
Let $\gamma\in(\gamma_0,1)$.
Since $\{\mu_j\}_{j=1}^\infty$ converges in the weak*-topology to $\mu_0$, 
 Portmanteau's theorem implies
 $\lim_{p\to\infty}\varlimsup_{j\to\infty} c_p(\mu_j)=0$.
From  this and \eqref{zwei},
there exists $p_0\geq0$ such that for each $p\geq p_0$ we have
    \begin{equation}\label{bra1}
   (1-c_p(\mu_{j}))\log(1-c_p(\mu_{j}))\geq-(1-\gamma) h(\mu_{j}),\end{equation}
 for sufficiently large $j$. Since $\int\phi d\mu_j>-\infty$ we have
\begin{equation}\label{have}
\begin{split}
\int\phi d\mu_j= \sum_{k=0}^p\phi(k)\mu_j[k]+\sum_{k=p+1}^\infty\phi(k) \mu_j[k] =\int\phi d\mu_j|_p-K_p(\mu_j).
\end{split}
\end{equation}
The equation \eqref{have} for $\mu_0$ in the place of $\mu_j$ implies 
$\int\phi d\mu_0|_{p}\to\int\phi d\mu_0$ as $p\to\infty$.
We assume $p$ is large enough so that $-\int\phi d\mu_0|_p>0$.
Since $\phi$ is bounded continuous on $\Sigma_p$,
the weak*-convergence of $\mu_{j}|_p$ to $\mu_0|_p$ as $j\to\infty$
gives $\int\phi d\mu_{j}|_{p}\to \int\phi d\mu_0|_{p}$.
In particular, 
for sufficiently large $j$ we have 
\begin{align*}
\frac{h(\mu_{j}|_p)}{-\int\phi d\mu_{j}|_{p}}&\geq\frac{h(\mu_{j})
-(\beta_\infty+\delta)K_p(\mu_j)+ (1-c_p(\mu_{j}))\log(1-c_p(\mu_{j}))}{-\int\phi d\mu_{j}-K_p(\mu_j)}\\
&\geq   
 \frac{\gamma h(\mu_{j})-(\beta_\infty+\delta)K_p(\mu_j)}{-\int\phi d\mu_{j}-K_p(\mu_j)}\\
 &\geq
  \gamma\frac{h(\mu_{j})}{-\int\phi d\mu_{j}}.\end{align*}
  The first inequality is a consequence of  Proposition \ref{approx-full}, 
  the second of \eqref{bra1} and \eqref{have}.
  The last inequality is trivial if $K_p(\mu_j)=0$. Otherwise
  we appeal to the following:
  for $a,b,c,d>0$ such that $c>d$ and $a/c\geq b/d$, $(a-b)/(c-d)\geq a/c$ holds.
  Apply this with $a=\gamma h(\mu_j)$, 
  $b=(\beta_\infty+\delta)K_p(\mu_j)$, $c=-\int\phi d\mu_j$, $d=K_p(\mu_j)$.
  The condition $a/c\geq b/d$ is fulfilled by virtue of \eqref{view} and $\gamma_0<\gamma<1$.
 
  The weak*-convergence of $\mu_{j}|_p$ to $\mu|_p$ as $j\to\infty$ takes place in the space of shift-invariant measures
on $\Sigma_{p}$ where the entropy is upper semi-continuous and $\phi$ is bounded continuous. Hence
\begin{equation}\label{bra4}
\frac{h(\mu_0|_{p})}{-\int\phi d\mu_0|_p}\geq \gamma\cdot \varlimsup_{j\to\infty} 
\frac{h(\mu_{j})}{-\int\phi d\mu_{j}}.\end{equation}
By Proposition \ref{approx-full} and \eqref{have} for $\mu_0$ in the place of $\mu_j$, we have $$
\lim_{p\to\infty}\frac{h(\mu_0|_{p})}{-\int\phi d\mu_0|_p}=\frac{h(\mu_0)}{-\int\phi d\mu_0}.$$
 Therefore, letting $p\to\infty$ and then $\gamma\to 1$ in \eqref{bra4} yields the desired inequality.
\medskip

\noindent{\it Case 2: $\phi$ is not constant on some $1$-cylinder of $X$.}
For each $q\geq1$, let $(E^q)^{\mathbb N}$ denote
 the full shift  with symbols in $E^q$ and
$\sigma_q$ the left shift on $(E^q)^{\mathbb N}$.
The map $$\theta\colon\{x_i\}_{i=0}^\infty\in (E^q)^{\mathbb N}\mapsto x_0x_1\cdots\in \mathbb N^{\mathbb N}$$
 is a homeomorphism onto its image and satisfies $\theta\circ \sigma_q=\sigma^q\circ\theta$.
Put 
$$X_q=\{\{x_i\}_{i=0}^\infty\in (E^q)^{\mathbb N}\colon \theta(x)\in X\}.$$
The restriction of $\theta$ to $X_q$ is also denoted by $\theta$.
The $1$-cylinder in $X_q$ corresponding to the symbol $\omega\in E^{q}$ is denoted by $(\omega)$.
Define $\Phi\colon X_q\to\mathbb R$  by
$$\Phi|_{(\omega)}=\sup_{ [\omega]}S_q\phi\quad\text{for each }\omega\in E^q.$$
Note that $\Phi$  is constant on each $1$-cylinder of  $X_q$.

\begin{lemma}\label{lconst}
For each $\mu\in\mathcal M(\sigma|_X)$ the following holds:

\begin{itemize}
\item[(a)] if $\int\phi d\mu>-\infty$ then
$\left|q\int\phi d\mu  -\int\Phi d(\mu\circ\theta)\right|\leq D_q(\phi);$
\item[(b)]  if $\int\Phi d(\mu\circ\theta)>-\infty$ then $\int\phi d\mu>-\infty$. 
\end{itemize}
\end{lemma}
\begin{proof}
Since $q\int\phi\circ\theta d(\mu\circ\theta)=\int (S_{q}\phi)\circ\theta d(\mu\circ\theta)$ we have
\begin{align*}
\left|q\int\phi d\mu -\int\Phi d(\mu\circ\theta)\right|&=\left|q\int\phi\circ\theta d(\mu\circ\theta) -\int\Phi d(\mu\circ\theta)\right|\\
&\leq\sup_{X_q}\left| (S_{q}\phi)\circ\theta-\Phi\right|\leq D_q(\phi),\end{align*}
as required in  Lemma \ref{lconst}(a).
Lemma \ref{lconst}(b) follows from approximating $\phi$ with $\mu$-integrable functions,
using Lemma \ref{lconst}(a) and then the Monotone Convergence Theorem.
\end{proof}


 The uniform continuity of $\phi$ implies
  \begin{equation}\label{varia}
 \sup_{\omega\in E^q}\sup_{x,y\in [\omega]}S_q\phi(x)-S_q\phi(y)=o(q)\quad (q\to\infty),\end{equation} 
 see \cite[Proposition 6.2(b)]{FieFieYur02}. From \eqref{varia} and Lemma \ref{lconst}(a),
 for any $\epsilon\in(0,-\int\phi d\mu_0)$
 there exists $q\geq1$ such that
  $ |\int\Phi d(\mu\circ\theta)-q\int\phi d\mu|\leq q\epsilon/2$
 holds for every $\mu\in\mathcal M_\phi(\sigma|_X)$.
Moreover,
$h(\mu\circ\theta)=qh(\mu)$ holds.
 By Lemma \ref{lconst},
 measures in the definition of $
P(\Phi)$ and those in the definition of $P(\phi)$ are in one-to-one correspondence.
 It follows that 
 $|P(\Phi)|\leq q\epsilon/2$
and
\begin{equation}\label{e}\left|P(\Phi)-\int \Phi d(\mu\circ\theta)+q\int\phi d\mu\right|\leq q\epsilon.\end{equation}
From \eqref{varia} we have
$\log Z_1(\beta\Phi)=o\left(1/q\right)+\log Z_{q}(\beta\phi)$
for every $\beta>0$,
which implies $\beta_\infty(\Phi)\leq\beta_\infty<1.$ 
Hence we obtain
\begin{align*}\frac{h(\mu_0)}{-\int\phi d\mu_0-\epsilon}&\geq
\frac{h(\mu_0\circ\theta)}{P(\Phi)-\int\Phi d(\mu_0\circ\theta)}\quad\text{
by \eqref{e} for $\mu=\mu_0$}\\
&\geq\varlimsup_{j\to\infty}\frac{h(\mu_j\circ\theta)}{P(\Phi)-\int\Phi d(\mu_j\circ\theta)}\\
&\geq\varlimsup_{j\to\infty}\frac{h(\mu_j)}{-\int\phi d\mu_j+\epsilon}
\quad\text{
by \eqref{e} for $\mu=\mu_j$}\\
&\geq\varliminf_{j\to\infty}\frac{-\int\phi d\mu_j}
{-\int\phi d\mu_j+\epsilon}\varlimsup_{j\to\infty}\frac{h(\mu_j)}{-\int\phi d\mu_j}.
\end{align*}
The second inequality follows from the result in Case 1 applied to $(X_q,\Phi)$.
Letting $\epsilon\to0$ yields \eqref{desired}.
This completes the proof of Theorem \ref{ups}.
\qed

\subsection*{Acknowledgments}
 This research was partially supported by the JSPS KAKENHI 19K21835, 20H01811.
I thank Johannes Jaerisch for fruitful discussions.

\end{document}